%% file: main.tex
\documentclass[letterpaper, 10 pt, conference]{ieeeconf}  

\IEEEoverridecommandlockouts                              
\usepackage{graphics} 
\usepackage{epsfig} 
\usepackage{mathptmx} 
\usepackage{times} 
\usepackage{amsmath} 
\usepackage{amssymb}  
\usepackage{bbm}

\usepackage{graphicx}               
\usepackage{color,xcolor}
\usepackage{url}
\usepackage{amsmath}
\usepackage{bbold}
\usepackage{mathtools}
\usepackage{tikz}
\usetikzlibrary{shapes,arrows}
\usepackage{balance}
\usepackage{authblk}

\DeclareSymbolFont{calletters}{OMS}{cmsy}{m}{n}
\DeclareSymbolFontAlphabet{\mathcal}{calletters}

\let\labelindent\relax
\usepackage{enumitem}

\newtheorem{lemma}{Lemma}
\newtheorem{problem}{Problem}
\newtheorem{assumption}{Assumption}

\newtheorem{theorem}{Theorem}
\newtheorem{remark}{Remark}

\newcommand{\modeSymbol}{\sigma}

\newcommand{\A}{\mathcal{A}}
\newcommand{\B}{\mathcal{B}}
\newcommand{\X}{\mathcal{X}}

\newcommand{\W}{\mathcal{W}}
\newcommand{\V}{\mathcal{V}}
\newcommand{\U}{\mathcal{U}}
\newcommand{\R}{\mathbb{R}}
\newcommand{\Z}{\mathcal{Z}}
\newcommand{\I}{\mathcal{I}}
\newcommand{\J}{\mathcal{J}}

\newcommand{\performanceSet}{\mathcal{Z}}

\renewcommand{\J}{\mathcal{J}}

\newcommand{\zeroOneSet}{\mathbb{B}}

\title{\LARGE \bf
Optimal Control for Linear Networked Control Systems with Information Transmission Constraints
}

\author[1]{Antoine Aspeel}
\author[2]{Kwesi Rutledge}
\author[1]{Raphaël M. Jungers}
\author[1]{Benoit Macq}
\author[2]{Necmiye Özay}

\affil[1]{ICTEAM at Université catholique de Louvain, \authorcr Email: {\tt \{antoine.aspeel, raphael.jungers, benoit.macq\}@uclouvain.be}}
\affil[2]{Department of Electrical Engineering and Computer Science at the University of Michigan - Ann Arbor, \authorcr Email: {\tt \{krutledg, necmiye\}@umich.edu}}

\begin{document}

\maketitle
\thispagestyle{empty}
\pagestyle{empty}

\begin{abstract}
This paper addresses the problem of robust control of a linear discrete-time system subject to bounded disturbances and to measurement and control budget constraints.

Using Q-parameterization and a polytope containment method, we prove that the co-design of an affine feedback controller, a measurement schedule and a control schedule can be exactly formulated as a mixed integer linear program with 2 binary variables per time step. As a consequence, this problem can be solved efficiently, even when an exhaustive search for measurement and control times would have been impossible in a reasonable amount of time.
\end{abstract}


\section{INTRODUCTION}
\input{sections/introduction}

\input{FIGS/block_diagram}

\section{PROBLEM STATEMENT}
\label{sec:problem_statements}
\input{sections/problems_statement}

\input{sections/related_problems}

\section{METHODS}
\label{sec:methods}
In this section, we show that a feasible solution of Problem \ref{prob:control_asap} can be obtained by solving a MILP with $2T$ binary variables. In Subsection \ref{sec:missing_Ff}, we show that missing measurements and controls can be modeled as linear indicator constraints on the gains $F_{(t,\tau)}$ and $f_t$. Subsection \ref{sec:Q_param} recalls the Q-parameterization approach which leads to express the trajectories of $z_t$ and $u_t$ as linear transformation of the uncertainties $w_t$, $v_t$ and $x_0$. It leads to the introduction of new design variables $Q$ and $r$. In Subsection \ref{sec:robustness} we use polytope containment techniques to handle the robustness constraints $z_t\in\performanceSet$ and $u_t\in\U$ for all uncertainties $w_t$, $v_t$ and $x_0$. This leads to linear constraints in terms of the new decision variables $Q$ and $r$. Then, it is shown in Subsection \ref{sec:missing_Qr} that the indicator constraints on $F_{(t,\tau)}$ and $f_t$ for missing measurements and controls can be expressed as linear indicator constraints in terms of the new design variables $Q$ and $r$. Finally, Subsection \ref{sec:mainResult} combines the previous results to prove that Problem \ref{prob:control_asap} is equivalent to a MILP.

\subsection{Missing measurements and controls}\label{sec:missing_Ff}
There are different ways to deal with missing measurements. A common way is to set the measurement matrix $C$ to zero when no measurements are taken, i.e., when $\modeSymbol^m_t=0$, and treat the system as time-varying \cite{jungers2018observability}. In our case, the measurement times are decision variables and such an approach would be equivalent to considering part of the dynamics of the system as a variable of the problem. This would lead to non-linearities.

For this reason, we deal with missing measurements in a different way. We consider that a measurement is available at each time step but that the measurements that should be missing are forbidden to the controller. For this reason, the gains associated with prohibited measurements are set to zero, i.e., $\modeSymbol^m_\tau=0$ implies $F_{(t,\tau)}=0$ for all $t$. Intuitively, we do not consider missing measurements, but prohibited measurements instead. Consequently, the choice of measurement times does not translate as a choice on the dynamics of the system, but as a constraint on the controller's gains. This idea is taken up in the following lemma.

\begin{lemma}[Forbidden measurements]\label{lemma:prohibited}
For given $x_0$, $w_t$, $v_t$, $\modeSymbol^m_t$ and $\modeSymbol^c_t$ for $t=0,\dots,T-1$, the sequences $\{z_t\}_{t=0}^T$ and $\{u_t\}_{t=0}^{T-1}$ generated by (\ref{eq:system_w_missing_meas}) and (\ref{eq:control_input}) are the same as the ones generated by
\begin{subequations}\label{eq:prohibited_dynamics}
    \begin{align}
    x_{t+1} & = A x_t + B u_t + w_t, \label{eq:prohibited_dynamics:x}\\
    y_t &= C x_t + v_t, \label{eq:prohibited_dynamics:y} \\
    z_t &= D x_t + d \label{eq:prohibited_dynamics:z}\\
    u_t &= f_t + \sum_{\tau\leq t} F_{(t,\tau)} y_\tau, \label{eq:prohibited_dynamics:u}
    \end{align}
\end{subequations}
when for all $\tau=0,\dots,T-1$,
\begin{equation}\label{eq:prohibited:measurement}
\modeSymbol^m_\tau=0 \Rightarrow \left( F_{(t,\tau)}=0, \text{\; for\ all\; } t=\tau,\dots T-1 \right),
\end{equation}
and for all $t=0,\dots,T-1$,
\begin{equation}\label{eq:prohibited:control}
\modeSymbol^c_t=0 \Rightarrow
\begin{cases}
f_t=f_{t-1}, \\
F_{(t,\tau)}=F_{(t-1,\tau)},\ \text{for}\ \tau=0,\dots,t
\end{cases}
\end{equation}
with $f_{-1}=0$ and $F_{(-1,\tau)}=F_{(t,t-1)}=0$.
\end{lemma}

\subsection{Trajectories and Q-parameterization}\label{sec:Q_param}
\input{sections/trajectories_Q_parametrization}

\subsection{Robustness constraints by polytope containment}\label{sec:robustness}
\input{sections/handlingRobustness}

\subsection{Missing measurements and controls revisited}\label{sec:missing_Qr}
\input{sections/binaryVariables}

\subsection{Main result}\label{sec:mainResult}
\input{sections/thm_control_ASAP}



\section{EXAMPLES}\label{sec:examples}

A \textsc{Julia} code using \textsc{Gurobi} \cite{gurobi} generating the figures and implementing the algorithms is available at {\scriptsize\url{https://github.com/kwesiRutledge/measurement-scheduling0/tree/master/examples/}}. The reported computation times have been obtained on a laptop with an \texttt{Intel(R) Core(TM) i5-7267U 3.1 GHz} processor.

\subsection{A Double Integrator Drone System}
\label{sec:drone}

\begin{figure}
\centering
\includegraphics[width=\columnwidth]{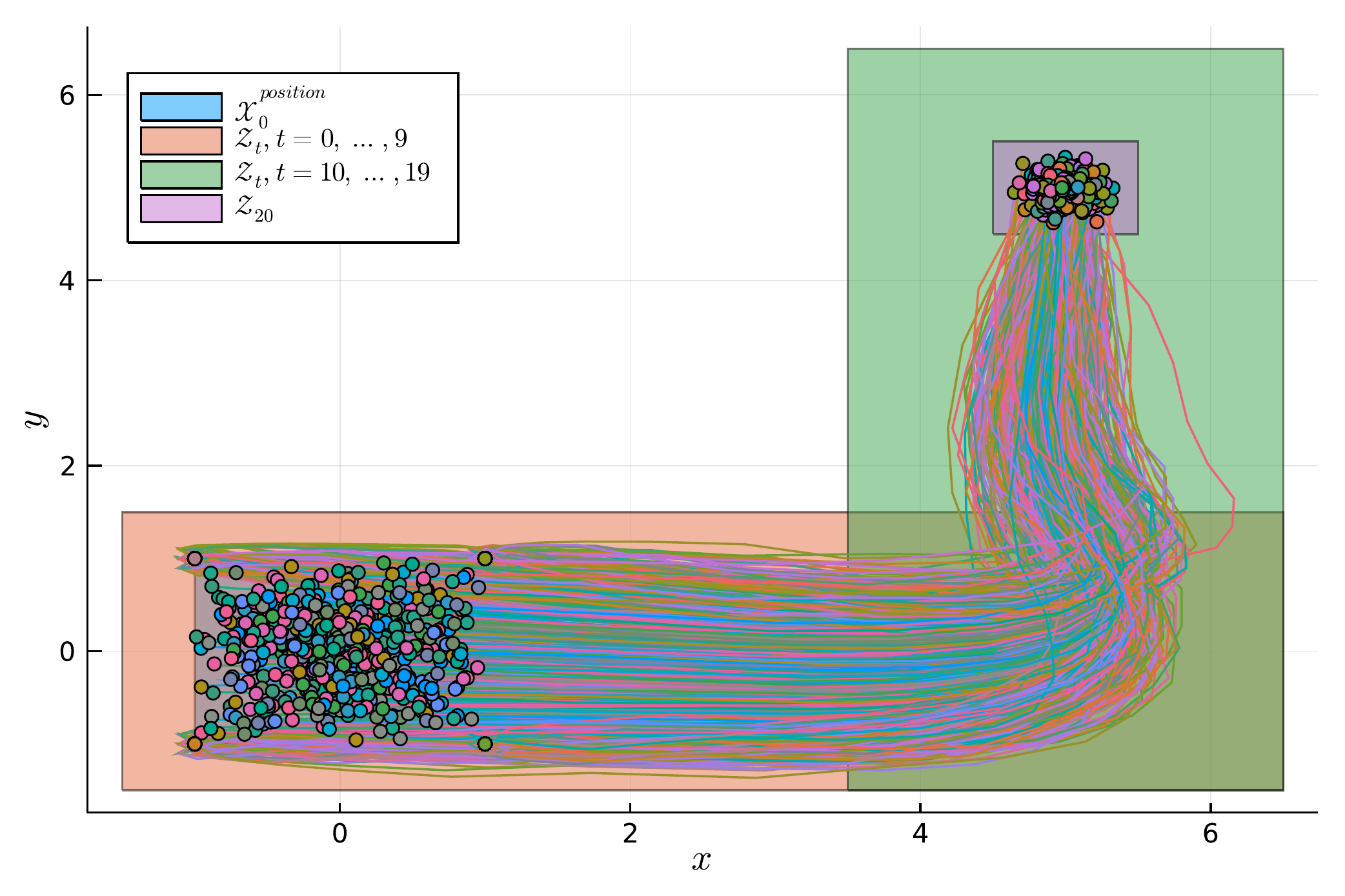}
\caption{Representation of 1000 trajectories $z_t$ sampled from the system described in Section \ref{sec:drone}. The admissible set of initial positions is indicated as $\X_0^{\text{position}}$ and the three different safety sets $\performanceSet_t$ are represented.}
\label{fig:drone_map}
\end{figure}

\begin{figure}
\centering
\includegraphics[width=\columnwidth]{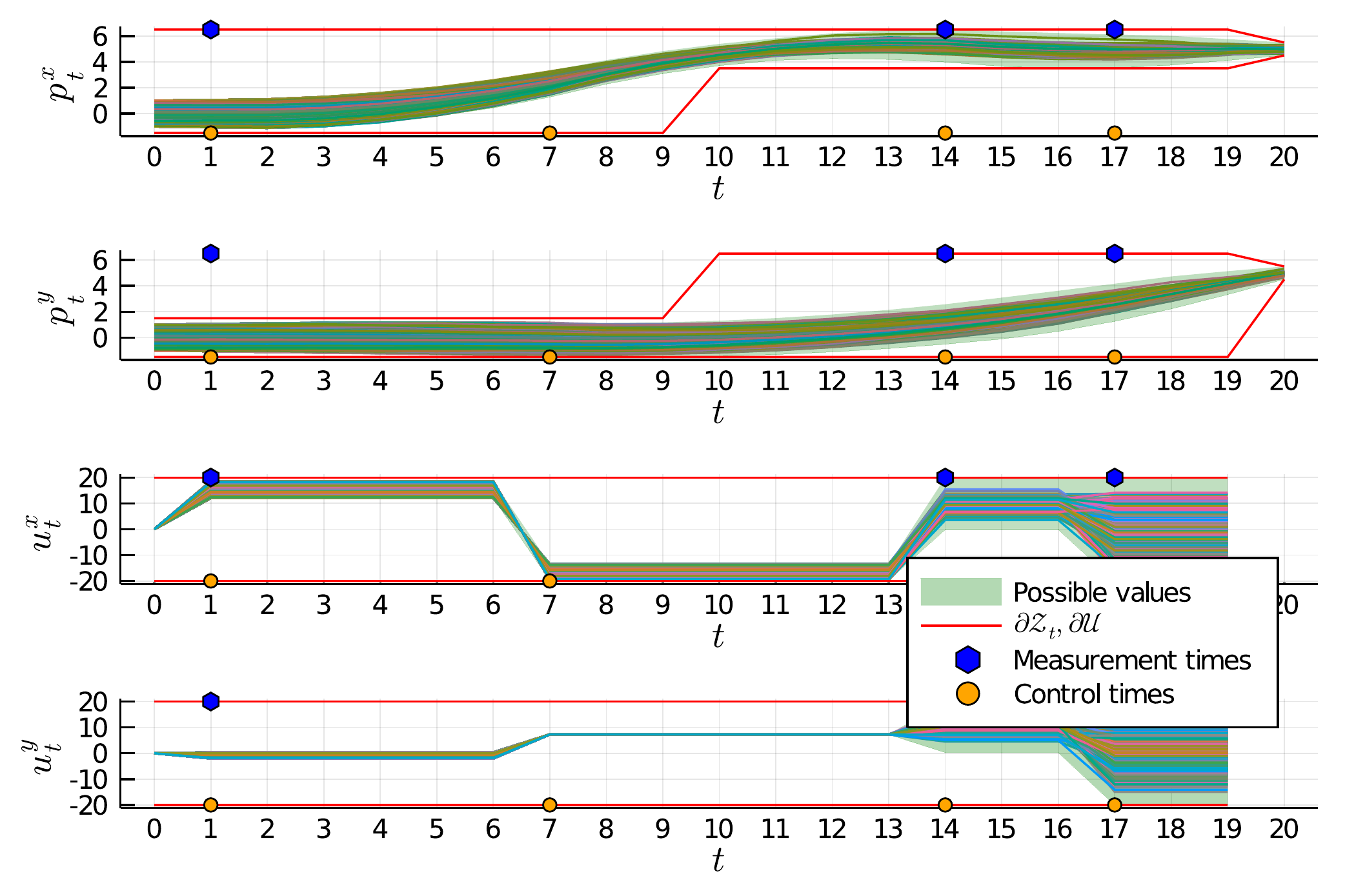}
\caption{Representation of 1000 simulated trajectories ($z_t$ and $u_t$) of the system described in Section \ref{sec:drone}. Measurement and control times are indicated. The boundaries of the safety sets $\partial\performanceSet_t$ and the boundaries of the set of admissible inputs $\partial\U$ are depicted. Finally, the sets of possible values for $z_t$ and $u_t$ for all possible noises are also depicted.}
\label{fig:drone_components}
\end{figure}

The task of remote monitoring of an area (e.g. for the remote reading of pressure guages in power plants) is one of the more recent topics of interest in cyber-physical systems \cite{bharadwaj2018synthesis}. Frequently implemented on collections of robots including drones, quadripeds, and more, remote monitoring controllers must be scalable enough to be implemented on large numbers of systems at once while also avoiding exchanging too many messages and burdening the network. It offers an excellent problem to solve with the methods developed in this paper because each robot's tasks can be encoded as reachability problems or reach-avoid problems (i.e. reach a target set while avoiding an unsafe set) which are readily handled with this method.

If one assumes that a drone's $x$ position $p^x_t$ and $y$ position $p^y_t$ are controlled by a simple force input, then the following dynamics may be written:
$$
\begin{array}{l}
    \ddot{p}^x_t = u^x_t, \\
    \ddot{p}^y_t = u^y_t.
\end{array}
$$
The state and the input of the system are respectively
$$
x_t = 
\begin{bmatrix}
    p^x_t & p^y_t & \dot{p}^x_t & \dot{p}^y_t
\end{bmatrix}^\top,\; u_t=
\begin{bmatrix}
u^x_t & u^y_t
\end{bmatrix}^\top,
$$
and the dynamics is
$$
\dot{x}_t =
\begin{bmatrix}
    0_{2\times 2} & I_2 \\
    0_{2\times 2} & 0_{2\times 2}
\end{bmatrix} x_t 
+
\begin{bmatrix}
    0_{2\times 2} \\ I_{2}
\end{bmatrix}
u_t.
$$
Finally, the output variable and the measurements are
$$
z_t=\begin{bmatrix}
p^x_t & p^y_t
\end{bmatrix}^\top,\; y_t=z_t+v_t.
$$

We consider an exactly discretized version of this system with a discretization time step of $0.1$. In addition, a process noise is considered. The time horizon is $T=20$ time steps with $N_m=3$ measurements and $N_c=4$ control inputs. In addition, we define $\W=[-0.05,0.05]^4$, $\V=[-0.05,0.05]^2$, $\X_0=[-1,1]^2\times \{0\}^2$ and $\U=[-20,20]^2$. Finally, the safety set $\performanceSet_t$ is time varying (see Remark \ref{remark:problemExtensions}) and takes 3 different values for respectively $t=0,\dots,9$; $t=10,\dots,19$; and $t=20$. They are represented in Fig.~\ref{fig:drone_map}. This problem is solved in 292 seconds.

Fig.~\ref{fig:drone_components} shows 1000 simulated trajectories. For half of them, the $x_0$, $w_t$ and $v_t$ are uniformly sampled in the polytopes. For the other half of the trajectories, each uncertain variable is randomly sampled among the vertices of the respective noise polytopes in order to promote extreme trajectories. The measurement and control times are indicated, as well as the boundaries of the safety sets $\partial\performanceSet_t$ and the boundaries of admissible inputs sets $\partial\U$. In addition, the set of possible values for $z_t$ and $u_t$ for all admissible uncertainty $w_t$, $v_t$ and $x_0$ are represented. The same trajectories are depicted in Fig.~\ref{fig:drone_map} where the set of admissible initial positions of the drone is indicated (and written $\X_0^{\text{position}}$), in addition to the time-varying safety sets.

We can verify that all simulated trajectories respect the constraints $z_t\in\Z$ and $u_t\in\U$. In addition, the measurement and control times are not regularly spaced. In fact, if we impose the measurement and control times to be spaced in time as regularly as possible, i.e., $\modeSymbol^m_t=1$ for $t\in\{0,10,19\}$ and $\modeSymbol^c_t=1$ for $t\in\{0,6,13,19\}$, then the problem is infeasible. 

Finally, the left part of Table \ref{tab:scalingAnalysis} presents the evolution of the solver time with respect to the number of drones to be controlled in parallel. The state dimension $n_x$ is also indicated. As expected, the solving time grows with the state dimension. Note that for a state space in 48 dimensions, the problem can be solved in less than $50$ minutes. In addition, these computations are executed offline which make the computation time not critical.

\begin{table}
\centering
\caption{Left: Problem \ref{prob:control_asap} is solved for the system described in Section \ref{sec:drone} with time horizon $T=20$ with a varying number of drones (and thus with varying state dimension $n_x$).  Right: Problem \ref{prob:control_asap} is solved for the system described in Section \ref{sec:lipm} with a varying time horizon $T$ (and the number of measurement and control times schedulings $\sigma^m$ and $\sigma^c$) with $N_m=N_c=T/2$ measurements and controls.}
\begin{tabular}{|c|c|c|}
\hline
Number & State & Solver\\
of Drones & Dimension & Time (s)\\
\hline
\hline
1 & 4 & 275.88 \\
\hline
4 & 16 & 332.74 \\
\hline
8 & 32 & 1118.96 \\
\hline
12 & 48 & 2846.47 \\
\hline
\end{tabular}$\ $
\begin{tabular}{|c|c|c|}
\hline
$T$ & Number & Solver\\
    & Schedules & Time (s) \\
\hline
\hline
2 & 4 & 0.03 \\
\hline
10 & 63504 & 0.81 \\
\hline
20 & $3.41\cdot 10^{10}$ & 19.96 \\
\hline
30 & $2.41\cdot 10^{16}$ & 3284.71 \\
\hline
\end{tabular}
\label{tab:scalingAnalysis}
\end{table}

\subsection{Planar Linear Inverted Pendulum Model (Simplified Walker Dynamics) \cite{posa2017balancing}}\label{sec:lipm}

\begin{figure}
\centering
\includegraphics[width=\columnwidth]{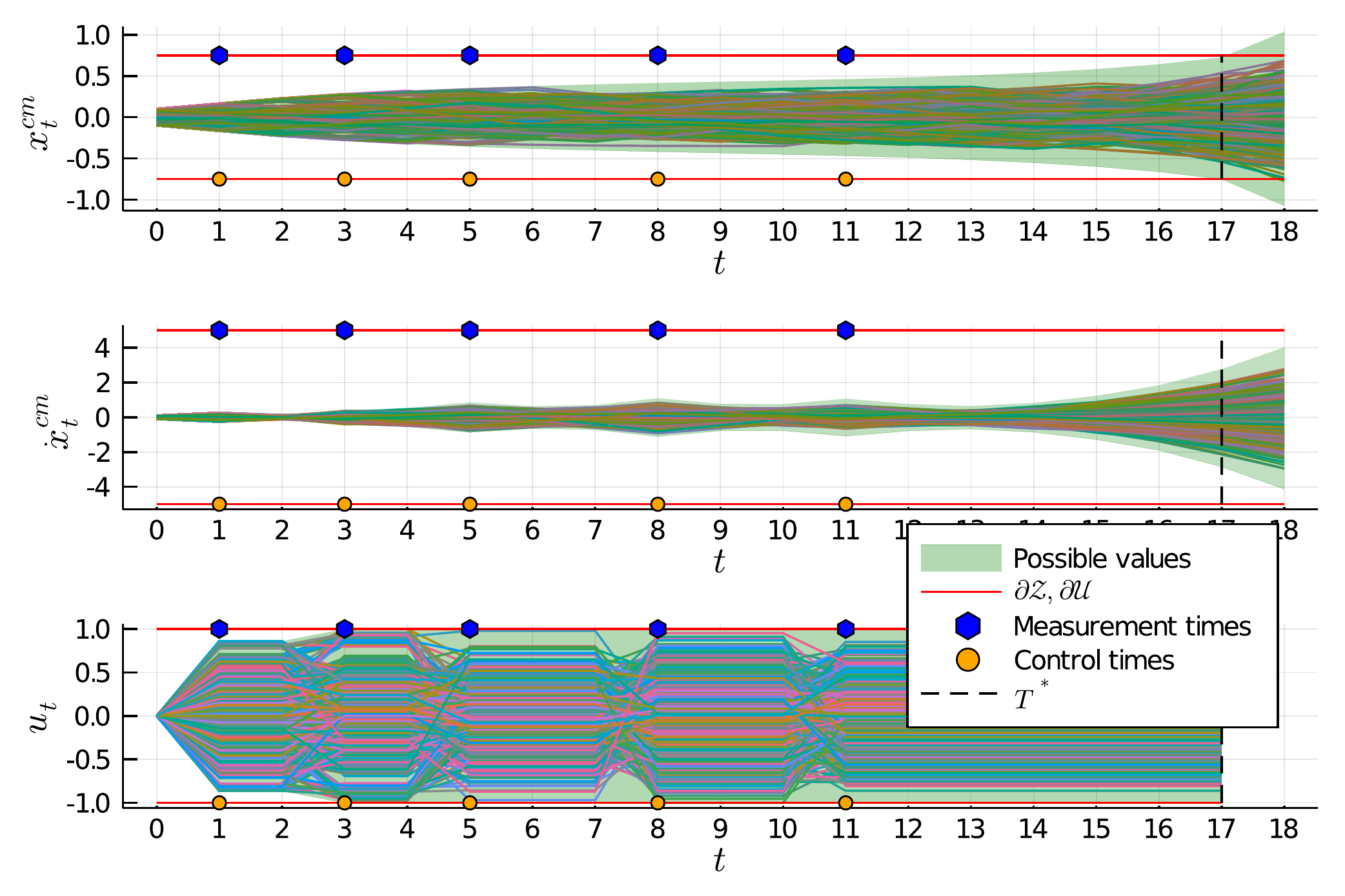}
\caption{Representation of 1000 simulated trajectories of system \eqref{eq:lipm:dynamics}. Measurement and control times are indicated. The boundaries of the safety set $\partial\performanceSet$ and the boundaries of the set of admissible inputs $\partial\U$ are depicted. Finally, the sets of possible values for $z_t$ and $u_t$ for all possibles noises are also depicted.}
\label{fig:lipm_components}
\end{figure}

The task of walking is one which typically does not require constant observation in order to safely execute. Humans routinely walk without knowledge of the exact position of their foot or angle of their torso, yet most controllers for walking robots require full state observation at all times.

To show the utility of our method, we analyze a small example where a robot model (the planar linear inverted pendulum) can be safely controlled by a policy which infrequently observes the state.

The planar linear inverted pendulum model as discussed in \cite{posa2017balancing} can be represented as a state space model with the following equations:
\begin{equation}\label{eq:lipm:dynamics}
    \ddot{x}^{cm} = \frac{g}{\bar{z}^{cm}} (x^{cm} + r^{foot} u),
\end{equation}
where $x^{cm}$ is the lateral position of the pendulum's center of mass, $\bar{z}^{cm}$ is the height of the pendulum's center of mass (which is assumed to be constant), $r^{foot}$ is the radius of the foot, and $u$ controls the center of pressure on the foot. We consider $\bar{z}^{cm}=1$, $r^{foot}=0.5$ and $g=9.81$.

We solve Problem \ref{prob:control_alap} for an exactly discretized version of \eqref{eq:lipm:dynamics} with a $0.1$ discretization time step. We have $x_t=\begin{bmatrix} x^{cm}_t & \dot{x}^{cm}_t \end{bmatrix}^\top$, $y_t=x_t$ and $z_t=x_t+v_t$. We consider $\W=[-0.05,0.05]^2$, $\V=[-0.01, 0.01]^2$, $\X_0=[-0.1,0.1]^2$, $\U=[-1,1]$, $\performanceSet=[-0.75,0.75]\times [-5,5]$, $N_m=N_c=5$. We consider the upper bound $\bar{T}=20$. This problem is solved in 284 seconds. The largest $T$ for which the problem is feasible is $T^*=17$. Fig.~\ref{fig:lipm_components} presents the obtained trajectories. For half of them, the uncertainties are randomly selected among the vertices of the polytopes to promote extreme trajectories.

For the same system, the right part of Table \ref{tab:scalingAnalysis} presents the evolution of the solving time for Problem \ref{prob:control_asap} according to the time horizon $T$ when the number of measurement and control times is $N_m=N_c=T/2$. In addition, the number of measurement and control scheduling is indicated (assuming that the budget constraints \eqref{eq:budget} are tight). This table shows that the solving time increases with the time horizon but much slower than the number of possible measurement and control times scheduling. Again, these computations are performed offline.

\begin{remark}
    Although these time horizons may be modest for many applications, a long time horizon $T'$ can be decomposed into multiple ``components" (e.g., $T_1 + T_2 = T'$) where smaller problems are solved with appropriate initial, intermediate conditions and budget constraints to derive a result for the long time horizon $T'$.
\end{remark}

\section{CONCLUSIONS AND FUTURE WORKS}\label{sec:conclusion_futureWorks}

In this paper, we have addressed the problem of co-designing a linear controller with memory, a measurement schedule, and a control schedule that guarantees safety. By proving that such a problem is equivalent to a MILP, we make this problem computationally tractable. This is illustrated on two examples.

For our future research, we want to adapt this method to the case where polytopes are zonotopes. In this case, we hope to reduce the computation time using methods such as zonotope order reduction techniques \cite{yang2018comparison}. Furthermore, we plan to develop an online version of this algorithm, where observation times are computed on the fly, while the system is running, and thus leveraging the current observations at each time step. In such a formulation, the problem will be solved again after each measurement in order to incorporate the information recently acquired.



\bibliographystyle{IEEEtran}
\bibliography{refs}

\appendices
\input{appendices/appendix_proofs}

\end{document}

%% file: sections/introduction.tex

There are a growing number of situations in the world where controllers of dynamical systems are required to use data acquired over a network to make their decisions. For example, consider smart electrical power grids, which aim to use sensing and prediction of power demand to control the power production of a country. The sensors of such a system would be distributed throughout homes, businesses, and public spaces. Thus, they will likely lie far away from the power production facility and direct connection of the sensors to the decision maker would be impossible. For such dynamical systems, controllers should take into account the properties of the network (e.g. packet drop probability, delay) when selecting their actions.

The areas of \emph{control over networks} and \emph{control of networks} were developed to address this design problem \cite{zhang2020networked}, where \emph{Control over networks} focuses on the design of controllers which are robust to the properties of a network while \emph{Control of Networks} focuses on the design of controllers for all of the nodes on a network such that a collective objective is achieved. In this work, we are interested in making formal guarantees about reachability of a target set when performing control and measurement of a single agent over a network. This objective aligns these results closely with the area of \emph{control over networks}.





\subsection{Related Work}

The obstacles to good performance while implementing a controller over a network are many. Packet losses are one pervasive problem and can be modelled by a probability of control or measurement packet loss during transmission \cite{schenato2007foundations}. Some results in the literature analyze the expected performance of controllers when the network is analyzed in this probabilistic manner. There may also exist delays in the transmission of information \cite[Remark II.4]{heemels2010networked} which can lead the same controller to be stable with one communication protocol/network architecture but be unstable within another. Limited bandwidth is also a problem, as a high frequency of controller and communication updates across the network may overburden it. So, in works like \cite{branicky2002scheduling,trivellato2010state}, adaptive communication protocols are developed to minimize the amount of information that must be sent across the network.

In the literature, two different types of controllers are used to minimize the number of actions (e.g. control or communication actions) that are transmitted by a controller during operation: event-triggered and self-triggered controllers \cite{heemels2012introduction}. In event-triggered controllers, the controller decides to take an action only when the measurement of the state satisfies a certain condition (e.g. state leaves a set). Note that in most event-triggered controllers, the state of the system or a measurement of it is available to the controller at all times \cite{zhang2017overview}, an assumption that we do not allow in this work.

In self-triggered controllers, the controller determines when it will take its next action while computing the current action. These controllers rely on the fact that the self-triggering control action always occurred simultaneously with a self-triggered measurement, which is a convenient but potentially restrictive assumption. This assumption is relaxed in the formulation of this work.

To our knowledge, the method defined in this work is the first derivation of robust output feedback control design with measurement and control budget constraints. Other areas, such as parsimonious control \cite{DEPERSIS2013parsimonious}, also have results defining how to minimize the number of control actions that are taken in a distributed system setting.

In the context of Q-parameterization, the quadratic invariance property has been studied. It has been shown that it is a necessary \cite{lessard2011quadratic} and sufficient \cite{rotkowitz2006characterization} condition for convexity. In this paper, we have similar results for linearity. Then, the combinatorial structure allowing to co-design the controller, the measurement and control times will lead to a \emph{Mixed Integer Linear Program} (MILP) and not only to a mixed integer convex program. Consequently, branch an bounds methods can be used to find an optimal solution efficiently.

Note that some alternatives to Q-parameterization exist, e.g., disturbance rejection control \cite{gao2006active}, or system level synthesis \cite{wang2019systemlevel}.

\subsection{Contribution}
In this paper, we develop a controller which satisfies constraints on the amount of bandwidth used on the network by jointly optimizing (i) a measurement schedule, (ii) a control schedule, and (iii) the controller's gains, while guaranteeing that the output variable remains in a safety set. The controller contains memory and a zero-order hold structure which is partially illustrated in Fig.~\ref{fig:block_diagram}. When the uncertainty sets in this problem are defined as polytopes, this problem can be formulated as a robust optimization using the \emph{polytope containment methods} discussed in \cite{BENTAL19991,sadraddini2020robust}. The output feedback controller, which would normally lead to complex nonconvex constraints can be parameterized linearly in terms of the optimization variables, using the methods of \emph{Q-Parameterization} \cite{skaf2010design,youla1976modern}. Then, the unique binary choices associated with the schedules introduce binary variables, making the optimization problem a MILP which is tractable to solve on practical problems.



\subsection{Paper outline}
We begin this paper by discussing the state of the art and the notation used throughout the paper. In Section \ref{sec:problem_statements}, the main problem we are interested in is defined, and several related problems are presented. In Section \ref{sec:methods}, we propose a solution to the problem using Q-Parameterization and Mixed-Integer Linear Programming. Section \ref{sec:examples} demonstrates our method on two examples and Section \ref{sec:conclusion_futureWorks} concludes and proposes future works. The proofs of the lemmas are in Appendix \ref{sec:proofs}.

\subsection{Notation and Terminology}
The set of real numbers is denoted by $\R$ and the set $\{0,1\}$ is denoted by $\zeroOneSet$. The set of $m\times n$ matrices with non-negative real entries is $\R^{m\times n}_+$. For a matrix $A$, $A^\top$ is the transpose of $A$. The symbol $I_n$ represents the identity matrix of size $n$, $0_{m\times n}$ is the zero $m\times n$ matrix, $\mathbb{1}_n$ is the $n$ dimensional vector of ones. For the sake of brevity, dimensions are sometimes omitted when they can be inferred from compatibility. The Kronecker product is $\otimes$. Inequalities between vectors are considered element-wise.
For $A\in\R^{m\times n}$ and $\I\subset \{1,\dots,m\}$, we write $A_{\I,:}$ the $|\I|\times m$ matrix whose rows correspond to the rows of $A$ with indices in $\I$. The notation $A_{:,\J}$ is used similarly for the columns. For integers $i<j$, we write $i:j=\{i,i+1,\dots,j\}$.

Calligraphic letters represent polytopes (except $\I$ and $\J$ which represent set of indices). For two polytopes $\A$ and $\B$, $\A\times \B$ is their Cartesian product; for a positive integer $n$, $\A^n$ is the $n$-th Cartesian power of $\A$; and $\partial\A$ denotes the boundary of $\A$.

%% file: FIGS/block_diagram.tex
\begin{figure}
\centering
\begin{tikzpicture}[scale=0.6]
\centering
\def\height{2}
\def\vbox{0.5}
\def\xSigma{4}
\def\xlim{6}
\draw (-2,-\height-\vbox) rectangle (2,-\height+\vbox); 
\draw (-\xlim-1,-\vbox) rectangle (\xlim+1,\vbox); 
\draw (-1.5,\height-\vbox) rectangle (1.5,\height+\vbox); 
\draw (-\xSigma-0.75,\height-\vbox) rectangle (-\xSigma+0.75,\height+\vbox); 
\draw (0,-\height) node{Controller};
\draw (0,0) node{Network};
\draw (0,\height) node{Plant};
\draw (-\xSigma,\height) node{ZOH};
\draw [->,>=latex] (-\xSigma+0.75,\height) --node[above]{$u_t$} (-1.5,\height);
\draw (\xSigma-0.5,\height) node{${\scriptstyle\bullet}$};
\draw (\xSigma+0.5,\height) node{${\scriptstyle\bullet}$};
\draw (1.5,\height) -- (\xSigma-0.5,\height) -- (\xSigma+0.5,\height+0.5);
\draw [->,>=latex] (\xSigma,\height+1) node[above]{$\modeSymbol^m_t$} -- (\xSigma,\height+0.25);
\draw (\xSigma+0.5,\height) -- (\xlim,\height) --node[above right]{$y_t$} (\xlim,\vbox);
\draw [dashed] (\xlim,\vbox) -- (\xlim,-\vbox);
\draw [->,>=latex] (\xlim,-\vbox) -- (\xlim,-\height) -- (2,-\height);
\draw (-\xSigma+0.5,-\height) node{ ${\scriptstyle\bullet}$};
\draw (-\xSigma-0.5,-\height) node{${\scriptstyle\bullet}$};
\draw [->,>=latex] (-\xSigma,-\height-1) node[below]{$\modeSymbol^c_t$} -- (-\xSigma,-\height+0.25);
\draw (-2,-\height) -- (-\xSigma+0.5,-\height) -- (-\xSigma-0.5,-\height+0.5);
\draw (-\xSigma-0.5,-\height) -- (-\xlim,-\height) -- (-\xlim,-\vbox);
\draw [dashed] (-\xlim,-\vbox) -- (-\xlim,\vbox);
\draw [->,>=latex] (-\xlim,\vbox) -- (-\xlim,\height) -- (-\xSigma-0.75,\height);
\draw [->,>=latex] (0,\height+\vbox+1) node[above]{$w_t,v_t$} --(0,\height+\vbox);
\end{tikzpicture}
\caption{Block diagram representing the interaction between the plant \eqref{eq:system_w_missing_meas} and the controller \eqref{eq:control_input}. ZOH means zero-order hold. 
} 
\label{fig:block_diagram}
\end{figure}

%% file: sections/problems_statement.tex
We consider the following discrete-time dynamical system
\begin{subequations}
    \label{eq:system_w_missing_meas}
    \begin{align}\label{eq:system_w_missing_meas:x}
        x_{t+1} & = A x_t + B u_t + w_t, \; \; w_t \in \W, \\
    	y_t & =
    	    \begin{cases}
    		    C x_t + v_t, & \modeSymbol^m_t = 1 \\
    			    \emptyset, & \modeSymbol^m_t = 0
    		\end{cases},
    		\; \; v_t \in \mathcal{V},\label{eq:system_w_missing_meas:y}\\
    	z_t &=D x_t + d. \label{eq:system_w_missing_meas:z}
    \end{align}
\end{subequations}
Quantities $w_t$, $v_t$, $x_0$ are unknown but are contained in known sets $\W$, $\V$, $\X_0$, respectively. The measurement scheduling signal $\modeSymbol^m_t \in \zeroOneSet $ determines if a measurement is acquired at time $t$.


We want to design the input signal $u_t$ from the previous measurements $y_t$. We restrict to linear controllers of the form
\begin{equation}\label{eq:control_input}
u_t = \begin{cases}
            f_t + \sum_{\tau\leq t \text{\ s.t.\ }\modeSymbol^m_\tau=1 } F_{(t,\tau)} y_\tau, & \modeSymbol^c_t=1,\\
            u_{t-1}, & \sigma^c_t=0,
      \end{cases}
\end{equation}
with $u_{-1}=0$. Column vectors $x_t$, $y_t$, $z_t$ and $u_t$ are respectively in $\R^{n_x}$, $\R^{n_y}$, $\R^{n_z}$ and $\R^{n_u}$. Other dimensions can be deduced from compatibility. The control scheduling signal $\modeSymbol^c_t\in \zeroOneSet $ determines when a new control input is sent to the plant.

\begin{remark}
When $\modeSymbol^c_t=0$, instead of zero-order hold, we could set the input to zero, i.e., $u_t=0$. Both options are presented in \cite{schenato2007foundations}, while \cite{heemels2010networked} focuses on zero-order hold.
\end{remark}

We are interested in controlling the output $z_t$ during a finite horizon of length $T$. Sensors and actuators are supposed to be constrained in their number of uses. This kind of situation occurs for example to save energy for sensors and actuators. Formally, we have a maximum number of measurements $N_m\geq0$, and a maximum number of new control inputs $N_c\geq0$. It is
\begin{equation}\label{eq:budget}
\sum_{t=0}^{T-1}\modeSymbol^m_t\leq N_m
\text{\; and \; }
\sum_{t=0}^{T-1}\modeSymbol^c_t\leq N_c.
\end{equation}

\begin{remark}
    These constraints can be replaced by any linear constraints, i.e., $\sum_{t=0}^{T-1}\left(c_{it}^m\modeSymbol_t^m + c_{it}^c\modeSymbol_t^c\right)\leq b_i$, for $i=1,\dots,I$ where $c_{it}^m$, $c_{it}^c$ and $b_i$ are known constants. For example, instead of considering two separated budgets, one for the measurements and the other for the controls, we could consider a common budget $\sum_{t=0}^{T-1}\left(c^m\modeSymbol^m_t + c^c\modeSymbol^c_t \right)\leq N$. Budgets over a sliding window of length $\bar{t}\geq 1$ can also be considered, i.e., $\sum_{\tau=t-\bar{t}}^{t-1}\modeSymbol_\tau\leq N$, for $t=\bar{t},\dots,T$.
\end{remark}

A safety set $\performanceSet$ is given for which our objective is to ensure $z_t\in\performanceSet$. In addition, we want the control input to stay inside a known set $\U$, i.e., $u_t\in\U$. Furthermore, we include the following assumption.
\begin{assumption}\label{assumption:polytopes}
    The sets $\X_0$, $\W$, $\V$, $\U$ and $\performanceSet$ are (not necessarily bounded) convex polyhedra.
\end{assumption}

It is assumed to have these polyhedra in H-representation. For $\A \in\{\X_0,\W,\V,\U,\performanceSet\}$, we write $\mathcal{A}=\{x|H_{\mathcal{A}} x\leq h_{\mathcal{A}} \}$ and $n_{\mathcal{A}}$ the number of rows in the matrix $H_{\mathcal{A}}$ and in the column vector $h_{\mathcal{A}}$.

The problem we are interested in is to find measurement times, control times, control gains and control offsets, that keep the output variable $z_t$ in the safety set $\Z$ during the complete horizon, i.e., for $t=0,\dots,T$.

\begin{problem}[Safety]
\label{prob:control_asap}
\begin{equation*}
\text{Find}\ \{\modeSymbol^m_t\}_{t=0}^{T-1},\ \{\modeSymbol^c_t\}_{t=0}^{T-1},\ \{f_t\}_{t=0}^{T-1},\ \{F_{(t,\tau)}\}_{t=0,\ \tau=0}^{T-1,\ t},
\end{equation*}
such that
\begin{itemize}
\item The dynamics \eqref{eq:system_w_missing_meas} hold,
\item The controller \eqref{eq:control_input} is used,
\item $\modeSymbol^m_t,\modeSymbol^c_t\in \zeroOneSet$ for $t=0,\dots,T-1$,
\item Budget constraints (\ref{eq:budget}) hold,
\item For all $w_t\in\W$, $v_t\in\V$,  for $t=0,\dots,T-1$, and for all $x_0\in\X_0$, it holds that $z_t\in\performanceSet$ for $t=0,\dots,T$, and $u_t\in\U$ for $t=0,\dots,T-1$.
\end{itemize}
\end{problem}


\begin{remark}
    \label{remark:problemExtensions}
    It is possible to add an objective function for minimizing a cost related to measurements and controls, e.g., $\sum_{t=0}^{T-1}\left(c_m\modeSymbol^m_t+c_c\modeSymbol^c_t \right)$, where $c_m$ and $c_c$ are given non negative constants.
    
    In addition, one could add the constraint $x_T\in\X_0$. Then, a solution to Problem \ref{prob:control_asap} can be applied periodically to keep $z_t\in\performanceSet$ and $u_t\in\U$ for all $t\geq0$.
\end{remark}

%% file: sections/related_problems.tex
\subsection{Related problem - Safety as long as possible}
In this problem, one wants to keep the output safe, i.e., $z_t\in\performanceSet$ as long as possible. We assume to have an \textit{a priori} upper bound $\bar{T}$ on the largest such time. This can be formalized as follows.
\begin{problem}[Safety as long as possible]\label{prob:control_alap}
\begin{equation*}
\underset{
            \small{
            \begin{array}{c}
                T\in\{0,\dots,\bar{T}\}\\
                \{\modeSymbol^m_t\}_{t=0}^{\bar{T}-1}, \{\modeSymbol^c_t\}_{t=0}^{\bar{T}-1}\\
                \{f_t\}_{t=0}^{\bar{T}-1}, \{F_{(t,\tau)}\}_{t=0,\ \tau=0}^{\bar{T}-1,\ t}
            \end{array} }
        }{\max} T,
\end{equation*}
such that
\begin{itemize}
\item Equations (\ref{eq:system_w_missing_meas}) and (\ref{eq:control_input}) hold for $t=0,\dots,\bar{T}-1$,
\item $\modeSymbol^m_t,\modeSymbol^c_t\in \zeroOneSet$ for $t=0,\dots,\bar{T}-1$,
\item Budget constraints (\ref{eq:budget}) hold with $\bar{T}$ instead of $T$,
\item For all $w_t\in\W$, $v_t\in\V$,  for $t=0,\dots,T-1$, and for all $x_0\in\X_0$, it holds that $z_t\in\performanceSet$ for $t=0,\dots,T$, and $u_t\in\U$ for $t=0,\dots,T-1$.
\end{itemize}
\end{problem}
Problem \ref{prob:control_alap} consists of finding the largest $T$ such that Problem \ref{prob:control_asap} is feasible. Note that if Problem \ref{prob:control_asap} is feasible for some $T'$, then it is also feasible for all $T\leq T'$. Then, a binary search algorithm can be used to find the optimal $T$ by solving $O\left(\log_2(\bar{T})\right)$ instances of Problem \ref{prob:control_asap}. In Subsection \ref{sec:lipm}, we solve Problem \ref{prob:control_alap} on an example.


%% file: sections/trajectories_Q_parametrization.tex
This section follows the presentation of the Q-parameterization from \cite{rutledge2020finite}. First, let's define the dynamical variables of \eqref{eq:prohibited_dynamics} in terms of trajectories:
\begin{equation}\label{eq:def:f}
\arraycolsep=3.5pt
\begin{array}{ll}
    x = \begin{bmatrix} x_0^\top & x_1^\top & \cdots & x_T^\top \end{bmatrix}^\top, 
    &z = \begin{bmatrix} z_0^\top & z_1^\top & \cdots & z_T^\top \end{bmatrix}^\top,\\
    u = \begin{bmatrix} u_0^\top & u_1^\top & \cdots & u_{T-1}^\top \end{bmatrix}^\top, 
    &w = \begin{bmatrix} w_0^\top & w_1^\top & \cdots & w_{T-1}^\top \end{bmatrix}^\top,\\ 
    v = \begin{bmatrix} v_0^\top & v_1^\top & \cdots & v_{T-1}^\top \end{bmatrix}^\top, 
    &f = \begin{bmatrix} f_0^\top & f_1^\top & \cdots & f_{T-1}^\top\\ \end{bmatrix}^\top, \\ 
    \mathbf{\sigma}^m = \begin{bmatrix} \mathbf{\sigma}^m_0 & \mathbf{\sigma}^m_1 & \cdots & \mathbf{\sigma}^m_{T-1} \end{bmatrix}^\top, 
    &\mathbf{\sigma}^c = \begin{bmatrix} \mathbf{\sigma}^c_0 & \mathbf{\sigma}^c_1 & \cdots & \mathbf{\sigma}^c_{T-1} \\ \end{bmatrix}^\top, 
\end{array}
\end{equation}
\begin{equation}\label{eq:def:F}
F = 
\begin{bmatrix}
    F_{(0,0)} & 0 & 0 & \cdots & 0 \\
    F_{(1,0)} & F_{(1,1)} & 0 & \cdots & 0 \\
    \vdots & \vdots & \vdots & \ddots & \vdots \\
    F_{(T-1,0)} & F_{(T-1,1)} & F_{(T-1,2)} & \cdots & F_{(T-1,T-1)} \\
\end{bmatrix}. 
\end{equation}
Note that $F$ and $f$ characterize the control input and are decision variables. The state trajectory $x$ (along with several other trajectories) is a nonlinear function of the decision variables $F$ and $f$:
\begin{equation}\label{eq:nonlinearEwpressionOfx}
\begin{array}{l}
    x = (H + SF(I-\bar{C}SF)^{-1} \bar{C} H)w + SF(I-\bar{C}SF)^{-1} v + \\
    \quad \quad (I + SF(I-\bar{C}SF)^{-1} \bar{C})Jx_0 + S(I+Q\bar{C}S)f,
\end{array}
\end{equation}
where
$$
J = 
\begin{bmatrix}
    I \\ A \\ A^2 \\ \vdots \\ A^{T}    
\end{bmatrix},\; 
H = 
\begin{bmatrix}
    0_{n_x\times n_x} & 0 & 0 & \cdots & 0 \\
    I & 0 & 0 & \cdots & 0 \\
    A & I & 0 & \cdots & 0 \\
    \vdots & \vdots & \vdots & \ddots & \vdots \\
    A^{T-1} & A^{T-2} & A^{T-3} & \cdots & I
\end{bmatrix},
$$
$$
\bar{C} = \begin{bmatrix}I_T \otimes C  & 0_{Tn_y\times n_x}\end{bmatrix},
$$
$$
S = 
\begin{bmatrix}
    0_{n_x\times n_u} & 0 & 0 & \cdots & 0 \\
    B & 0 & 0 & \cdots & 0 \\
    AB & B & 0 & \cdots & 0 \\
    \vdots & \vdots & \vdots & \ddots & \vdots \\
    A^{T-1}B & A^{T-2}B & A^{T-3}B & \cdots & B
\end{bmatrix}.
$$

Thus, in view of \eqref{eq:nonlinearEwpressionOfx}, efficiently searching the set of feasible gains ($F$ and $f$) for Problem \ref{prob:control_asap} would be a search over a nonconvex set. However,
from \cite[Theorem 2, Equation (17)]{rutledge2020finite}, the following also holds
\begin{equation}\label{eq:elsevier1}
x=(H+SQ\bar{C}H)w+SQv+(I+SQ\bar{C})Jx_0+Sr,
\end{equation}
\begin{equation}\label{eq:elsevier2}
u=Q\bar{C}Hw+Qv+Q\bar{C}Jx_0+r,
\end{equation}
where the following Q-parameterization mapping is used
\begin{align}\label{eq:Q_param:Q}
    Q&=F(I-\bar{C}SF)^{-1},\\
    r&=(I+Q\bar{C}S)f,\label{eq:Q_param:r}
\end{align}
and gives a $n_x\times n_y$ block lower triangular matrix $Q$. Conversely, if $Q$ is $n_x\times n_y$ block lower triangular, then this mapping is invertible and the inverse mapping is
\begin{align}\label{eq:Q_param:F}
F&=(I+Q\bar{C}S)^{-1}Q, \\
f&=(I+Q\bar{C}S)^{-1}r.\label{eq:Q_param:f}
\end{align}
This mapping allows to express our problem in terms of $Q$ and $r$ instead of $F$ and $f$.

In addition, let's write
$$
\bar{D}=I_{T+1}\otimes D, \; \bar{d}=\mathbb{1}_{T+1}\otimes d.
$$
Then, using $z=\bar{D}x+\bar{d}$, (\ref{eq:elsevier1}) and (\ref{eq:elsevier2}), one can write
\begin{equation}\label{eq:P:main}
  \begin{bmatrix}
      z	\\	u
\end{bmatrix}
  =
  \begin{bmatrix}
     P_{z w}	&	P_{z v}  & P_{z x_0}\\
        P_{u w}	&	P_{u v}  & P_{u x_0}
  \end{bmatrix}
  \begin{bmatrix}
     w	\\	v  \\  x_0
\end{bmatrix}
  +
  \begin{bmatrix}
      \tilde{z}	\\	\tilde{u}
\end{bmatrix},
\end{equation}
where
\begin{align}\label{eq:P:submatrices}
&
\arraycolsep=2pt
\begin{array}{rlrlrl}
	P_{z w} = & \bar{D}(H+SQ\bar{C}H), &
    P_{z v} = & \bar{D}SQ, &
    P_{z x_0}=& \bar{D}(I+SQ\bar{C})J, \\
    P_{u w} = & Q\bar{C}H, &
    P_{u v} = & Q, &
    P_{u x_0}=& Q\bar{C}J,
\end{array}\\
&\begin{array}{lcr}
\tilde{z}=\bar{D}Sr+\bar{d} &\text{and}& \tilde{u}=r.
\end{array}\label{eq:x_u_tilde}
\end{align}
Note that these quantities depend linearly on the new decision variables $Q$ and $r$, which is at the heart of Q-parameterization.

%% file: sections/handlingRobustness.tex
Problem \ref{prob:control_asap} contains the robust constraints $z_t\in\performanceSet$ and $u_t\in\U$ for all $w_t$, $v_t$ and $x_0$. To deal with such constraints, we use polytope containment techniques. To this end, we will need the following extension of the Farkas' lemma.
\begin{lemma}[H-Polytope in H-Polytope, \cite{hennet1989extension}]
    \label{lem:hpoly_in_hpoly}
    Let $\mathcal{A} = \{ x \in \mathbb{R}^n \; | \; H_\mathcal{A} x \leq h_\mathcal{A} \}$, $\mathcal{B} = \{ x \in \mathbb{R}^n \; | \; H_\mathcal{B} x \leq h_\mathcal{B} \} \subset \mathbb{R}^n$, $H_\mathcal{A} \in \mathbb{R}^{n_{\mathcal{A}} \times n}$, $H_\mathcal{B} \in \mathbb{R}^{n_{\mathcal{B}} \times n}$. We have $\mathcal{A} \subseteq \mathcal{B}$ if and only if
    \begin{equation*}
        \exists \Lambda \in \mathbb{R}^{n_{\mathcal{B}} \times n_{\mathcal{A}}}_+ \text{ such that }
        \Lambda H_\mathcal{A} = H_\mathcal{B}, \; \Lambda h_\mathcal{A} \leq h_\mathcal{B}.
    \end{equation*}
\end{lemma}

We use Lemma \ref{lem:hpoly_in_hpoly} to prove the following lemma which allows to handle the robustness constraint $z_t\in\performanceSet$.

\begin{lemma}[Safety constraint]\label{lemma:robustness_constraint}
Let the sequence $\{z_t\}_{t=0}^T$ be generated by (\ref{eq:prohibited_dynamics}). The inclusion $z_t\in\performanceSet$ holds for all $t=0,\dots,T$, for all $w_t\in\W$, for all $v_t\in\V$ and for all $x_0\in\X_0$ if and only if there exists $\Lambda\in\mathbb{R}^{(T+1)n_\performanceSet\times (T(n_\W+n_\V)+n_{\X_0})}_+$ such that
\begin{equation}\label{eq:robustness_constraint:1} \setlength\arraycolsep{3pt}
\Lambda \begin{bmatrix} I_T \otimes H_{\W} & 0 & 0 \\ 0 & I_T \otimes H_{\V} & 0 \\ 0 & 0 & H_{\X_0} \end{bmatrix} = \left(I_{T+1}\otimes H_\performanceSet\right) \begin{bmatrix} P_{z w} & P_{z v} & P_{z x_0} \end{bmatrix},
\end{equation}
and
\begin{equation}\label{eq:robustness_constraint:2}
\Lambda \begin{bmatrix} \mathbb{1}_T \otimes h_{\W} \\ \mathbb{1}_T \otimes h_{\V} \\ h_{\X_0} \end{bmatrix} \leq \mathbb{1}_{T+1}\otimes h_\performanceSet - \left(I_{T+1}\otimes H_\performanceSet\right)\tilde{z},
\end{equation}
where $P_{z w}$, $P_{z v}$, $P_{z x_0}$ are defined in (\ref{eq:P:submatrices}) and $\tilde{z}$ is defined in (\ref{eq:x_u_tilde}).
\end{lemma}

We can state a similar result for the bounded input constraint $u_t\in\U$.

\begin{lemma}[Bounded inputs]\label{lemma:bounded_inputs}
Let the sequence $\{u_t\}_{t=0}^{T-1}$ be generated by (\ref{eq:prohibited_dynamics}). The inclusion $u_t\in\U$ holds for all $t=0,\dots,T-1$, for all $v_t\in\V$, for all $w_t\in\W$ and for all $x_0\in\X_0$ if and only if there exists $\Gamma\in\mathbb{R}^{Tn_\U\times (T(n_\W+n_\V)+n_{\X_0})}_+$ such that
\begin{equation}\label{eq:bounded_inputs:1} \setlength\arraycolsep{4pt}
\Gamma \begin{bmatrix} I_{T} \otimes H_{\W} & 0 & 0 \\ 0 & I_T \otimes H_{\V} & 0 \\ 0 & 0 & H_{\X_0} \end{bmatrix} = \left(I_T\otimes H_\U\right) \begin{bmatrix} P_{u w} & P_{u v} & P_{u x_0} \end{bmatrix},
\end{equation}
and
\begin{equation}\label{eq:bounded_inputs:2}
\Gamma \begin{bmatrix} \mathbb{1}_T \otimes h_{\W} \\ \mathbb{1}_T \otimes h_{\V} \\ h_{\X_0} \end{bmatrix} \leq \mathbb{1}_T\otimes h_\U - \left(I_T\otimes H_\U\right)\tilde{u},
\end{equation}
where $P_{u w}$, $P_{u v}$, $P_{u x_0}$ are defined in (\ref{eq:P:submatrices}) and $\tilde{u}$ is defined in (\ref{eq:x_u_tilde}).
\end{lemma}


\begin{remark}
    Note that Lemmas \ref{lemma:robustness_constraint} and \ref{lemma:bounded_inputs} give linear constraints on $P_{zw}$, $P_{zv}$, $P_{zx_0}$, $P_{uw}$, $P_{uv}$ and $P_{ux_0}$ which are linear functions of $Q$ and $r$. Consequently, these constraints are linear in the decision variables $Q$ and $r$. Thus, a feasibility problem which once contained non-convex constraints on the decision variables $F$ and $f$ can be transformed into an equivalent feasibility problem with convex constraints on the decision variables $Q$ and $r$.
\end{remark}

%% file: sections/binaryVariables.tex
The two following Lemmas state how the forbidden measurements constraint (\ref{eq:prohibited:measurement}) and the missing control constraint \eqref{eq:prohibited:control} can be expressed linearly in terms of the new design variables $Q$ and $r$.

\begin{lemma}[Missing measurements]\label{lemma:missing_measurements}
Let a binary measurement signal $\modeSymbol^m \in \zeroOneSet^T$ and let $\{F_{(t,\tau)}\}_{t=0,\ \tau=0}^{T-1,\ t}$ be some control gains. Then, (\ref{eq:prohibited:measurement}) holds for $\tau=0,\dots,T-1$, if and only if
\begin{equation}
    \label{eq:measurement_constraint}
    \modeSymbol^m_\tau=0 \Rightarrow Q_{:,\J_\tau}=0, \text{\ for\ } \tau=0,\dots,T-1,
\end{equation}
where $Q$ is defined by (\ref{eq:def:F}) and (\ref{eq:Q_param:Q}) and where $\J_\tau=1+\tau n_y : (\tau+1) n_y$.
\end{lemma}

\begin{lemma}[Missing controls]\label{lemma:missing_controls}
    Let a binary control signal $\modeSymbol^c\in\zeroOneSet^T$ and let $\{f_t\}_{t=0}^{T-1}$ and $\{F_{(t,\tau)}\}_{t=0,\ \tau=0}^{T-1,\ t}$ be some control offsets and control gains. Then, (\ref{eq:prohibited:control}) holds for all $t=0,\dots,T-1$ if and only if the following indicator constraints hold,
    \begin{equation}\label{eq:control_constraint}
        \text{for\; } t=0,\dots,T-1,\ 
        \modeSymbol^c_t=0 \Rightarrow
        \begin{cases}
        Q_{\I_t,:}=Q_{\I_{t-1},:}, \\
        r_{\I_t}=r_{\I_{t-1}},
        \end{cases}
    \end{equation}
    where $\I_t=1+t n_u : (t+1) n_u$, $Q_{\I_{-1},:}=0$ and $r_{\I_{-1}}=0$.
\end{lemma}

Similar results can be obtained with quadratic invariance. Remark that Lemmas \ref{lemma:missing_measurements} and \ref{lemma:missing_controls} prove that the constraints are linear (and not just convex) in $Q$ and $r$.




%% file: sections/thm_control_ASAP.tex
Before stating our main result, recall that indicator constraints such that (\ref{eq:measurement_constraint}) and (\ref{eq:control_constraint}) can be handled as mixed integer linear constraints thanks to the Big-M formulation. The Big-M formulation uses the fact that for $\sigma\in\{0,1\}$, the indicator constraint $\sigma=0\Rightarrow f(x)=0$ is equivalent to $-\sigma \bar{M}\leq f(x) \leq \sigma\bar{M}$ for $\bar{M}$ sufficiently large. 

\begin{theorem}\label{thm:asap}
Problem \ref{prob:control_asap} is feasible if and only if the following MILP is feasible.

\begin{problem}
\begin{equation*}
\text{Find}\  \modeSymbol^m,\ \modeSymbol^c,\ Q,\ r,\ \Lambda\geq 0,\ \Gamma\geq 0,
\end{equation*}
such that
\begin{itemize}
\item $\modeSymbol^m_t,\modeSymbol^c_t\in\{0,1\}$ for $t=0,\dots,T-1$,
\item Budget constraints (\ref{eq:budget}) hold,
\item Matrix $Q$ is $n_x\times n_y$ block lower triangular,
\item Safety constraints (\ref{eq:robustness_constraint:1}) and (\ref{eq:robustness_constraint:2}) hold,
\item Bounded input constraints (\ref{eq:bounded_inputs:1}) and (\ref{eq:bounded_inputs:2}) hold,
\item Measurement compatibility constraint (\ref{eq:measurement_constraint}) holds,
\item Control compatibility constraint (\ref{eq:control_constraint}) holds.
\end{itemize}
\end{problem}
Moreover, the feasible values of $\modeSymbol^m$ and $\modeSymbol^c$ are the same for both problems and the feasible $Q$ and $r$ are linked by transformations (\ref{eq:Q_param:F}) and (\ref{eq:Q_param:f}) to the feasible values of $\{f_t\}_{t=0}^{T-1}$ and $\{F_{(t,\tau)}\}_{t=0,\ \tau=0}^{T-1,\ t}$ for Problem \ref{prob:control_asap}.
\end{theorem}

\begin{proof}
From Lemma \ref{lemma:prohibited}, the sequences $\{z_t\}_{t=0}^T$ and $\{u_t\}_{t=0}^{T-1}$ generated by (\ref{eq:system_w_missing_meas}) and (\ref{eq:control_input}) are the same than the ones generated by (\ref{eq:prohibited_dynamics}) when (\ref{eq:prohibited:measurement}) and (\ref{eq:prohibited:control}) hold.

Then, Lemma \ref{lemma:robustness_constraint} indicates that (\ref{eq:prohibited_dynamics}) and the robustness constraint $z_t\in\performanceSet$ for $t=0,\dots,T$ hold if and only if (\ref{eq:robustness_constraint:1}) and (\ref{eq:robustness_constraint:2}) hold. Similarly, Lemma \ref{lemma:bounded_inputs} shows that the constraint $u_t\in\U$ for $t=0,\dots,T-1$ holds if and only if (\ref{eq:bounded_inputs:1}) and (\ref{eq:bounded_inputs:2}) hold.

From Lemma \ref{lemma:missing_measurements}, (\ref{eq:prohibited:measurement}) holds if and only if (\ref{eq:measurement_constraint}) holds, and thanks to Lemma \ref{lemma:missing_controls}, (\ref{eq:prohibited:control}) holds if and only if (\ref{eq:control_constraint}) holds. Finally, the indicator constraints can be handled using the Big-M method.
\end{proof}

Solving such a MILP can be done efficiently on many practical situations using a branch and bound approach. This is illustrated in Section \ref{sec:examples}. In addition, this problem is solved offline, before any measurements are received. Then, the feasible solution can be efficiently used online.


%% file: appendices/appendix_proofs.tex
\section{Proofs}\label{sec:proofs}

\subsection{Proof of Lemma \ref{lemma:prohibited}}

If $\modeSymbol^c_t=1$, one can use (\ref{eq:prohibited:measurement}) to write
$$
f_t+\sum_{\tau\leq t \text{\ s.t.\ }\modeSymbol^m_\tau=1 } F_{(t,\tau)} y_\tau = f_t + \sum_{\tau\leq t} F_{(t,\tau)} y_\tau,
$$
then \eqref{eq:control_input} and \eqref{eq:prohibited_dynamics:u} give the same $u_t$. On the other hand, if $\modeSymbol^c_t=0$, one can use \eqref{eq:prohibited:control} to write
\begin{align*}
f_t+\sum_{\tau\leq t}F_{(t,\tau)} y_\tau &= f_{t-1}+\sum_{\tau\leq t}F_{(t-1,\tau)} y_\tau \\
    &= f_{t-1}+\sum_{\tau\leq t-1}F_{(t-1,\tau)} y_\tau.
\end{align*}
then \eqref{eq:control_input} and \eqref{eq:prohibited_dynamics:u} give the same $u_t$.

It follows that the sequences of $u_t$ are the same. But then, so it is for the sequences of $x_t$ and finally, for the sequences of $z_t$.

\subsection{Proof of Lemma \ref{lemma:robustness_constraint}}

By writing $\eta=\begin{bmatrix} w^\top	&	v^\top  & x_0^\top \end{bmatrix}^\top\in\W^T\times\V^T\times\X_0$ and $P_{z,:}=\begin{bmatrix} P_{z w} & P_{z v} & P_{z x_0} \end{bmatrix}$ and using \eqref{eq:P:main}, we can write
$$
\begin{array}{lrcl}
&z_t & \in& \performanceSet \text{\ for\ all\ } t=0,\dots,T\\
\Leftrightarrow&H_\performanceSet z_t &\leq& h_\performanceSet \text{\ for\ all\ } t=0,\dots,T\\
\Leftrightarrow&\left(I_{T+1}\otimes H_\performanceSet\right) z &\leq& \mathbb{1}_{T+1}\otimes h_\performanceSet\\
\Leftrightarrow&\left(I_{T+1}\otimes H_\performanceSet\right)(P_{z,:}\eta+\tilde{z}) &\leq& \mathbb{1}_{T+1}\otimes h_\performanceSet\\
\Leftrightarrow&\left(I_{T+1}\otimes H_\performanceSet\right)P_{z,:}\eta &\leq& \mathbb{1}_{T+1}\otimes h_\performanceSet\\
&&&\hspace{1cm}- \left(I_{T+1}\otimes H_\performanceSet\right)\tilde{z}.
\end{array}
$$
The last inequality can be interpreted as requiring that $\eta$ is in a polyhedron (let's call it $\mathcal{P}$). Then, $z_t\in\performanceSet$ for all $t=0,\dots,T$ and for all $\eta\in\W^T\times\V^T\times\X_0$ if and only if $\eta\in\mathcal{P}$ for all $\eta$. This is equivalent to $\W^T\times\V^T\times\X_0 \subseteq \mathcal{P}$. Then, using Lemma \ref{lem:hpoly_in_hpoly}, the constraint $z_t\in\performanceSet$ for all $t=0,\dots,T$ and for all $\eta\in\W^T\times\V^T\times\X_0$ holds if and only if there exists $\Lambda\in\mathbb{R}^{(T+1)n_\performanceSet\times (T(n_\W+n_\V)+n_{\X_0})}_+$ such that (\ref{eq:robustness_constraint:1}) and (\ref{eq:robustness_constraint:2}) hold.

\subsection{Proof of Lemma \ref{lemma:bounded_inputs}}

The proof is similar to the one of Lemma \ref{lemma:robustness_constraint}. Let $\eta=\begin{bmatrix} w^\top & v^\top & x_0^\top \end{bmatrix}^\top\in\W^T\times\V^T\times\X_0$ and $P_{u,:}=\begin{bmatrix} P_{u w} & P_{u v} & P_{u x_0} \end{bmatrix}$ and using \eqref{eq:P:main}, we can write
$$
\begin{array}{lrcl}
&u_t & \in& \U \text{\ for\ all\ } t=0,\dots,T-1\\
\Leftrightarrow&H_\U u_t &\leq& h_\U \text{\ for\ all\ } t=0,\dots,T-1\\
\Leftrightarrow&\left(I_{T}\otimes H_\U\right) u &\leq& \mathbb{1}_{T}\otimes h_\U\\
\Leftrightarrow&\left(I_{T}\otimes H_\U\right)(P_{u,:}\eta+\tilde{u}) &\leq& \mathbb{1}_{T}\otimes h_\U\\
\Leftrightarrow&\left(I_{T}\otimes H_\U\right)P_{u,:}\eta &\leq& \mathbb{1}_{T}\otimes h_\U- \left(I_{T}\otimes H_\U\right)\tilde{u}.
\end{array}
$$
The last inequality can be interpreted as requiring that $\eta$ is in a polytope (let's call it $\mathcal{Q}$). Then, $u_t\in\U$ for all $t=0,\dots,T-1$ and for all $\eta\in\W^T\times\V^T\times\X_0$ if and only if $\eta\in\mathcal{Q}$ for all $\eta$. This is equivalent to $\W^T\times\V^T\times\X_0 \subseteq \mathcal{Q}$. Then, using Lemma \ref{lem:hpoly_in_hpoly}, the constraint $u_t\in\U$ for all $t=0,\dots,T-1$ and for all $\eta\in\W^T\times\V^T\times\X_0$ holds if and only if there exists $\Gamma\in\mathbb{R}^{Tn_\U\times (T(n_\W+n_\V)+n_{\X_0})}_+$ such that (\ref{eq:bounded_inputs:1}) and (\ref{eq:bounded_inputs:2}) hold.

\subsection{Proof of Lemma \ref{lemma:missing_measurements}}

Let $\tau\in\{0,\dots,T-1\}$ and assume $\modeSymbol^m_\tau=0$. From (\ref{eq:def:F}), we have that for a fixed $\tau$, the matrices $F_{(t,\tau)}$ are in the same columns of $F$ for all $t$. The indices of these columns are $\J_\tau$. Then, $F_{(t,\tau)}=0$ for all $t=\tau,\dots,T-1$, if and only if $F_{:,\J_\tau}=0$. Then, looking at (\ref{eq:Q_param:F}), one can write
$$
F_{:,\J_\tau}=(I+Q\bar{C}S)^{-1}Q_{:,\J_\tau}.
$$
Observing that $(I+Q\bar{C}S)^{-1}$ is invertible, it follows that $F_{:,\J_\tau}=0$ if and only if $Q_{:,\J_\tau}=0$.

\subsection{Proof of Lemma \ref{lemma:missing_controls}}

Let $t\in\{1,\dots,T-1\}$ and assume that $\modeSymbol^c_t=0$. From \eqref{eq:def:f} and \eqref{eq:def:F}, the constraints $f_t = f_{t-1}$ and $F_{(t,\tau)} = F_{(t-1,\tau)}$ can be written $f_{\I_t}=f_{\I_{t-1}}$ and $F_{\I_t,:}=F_{\I_{t-1},:}$.

Otherwise, from (\ref{eq:Q_param:Q}), one can write
\begin{equation}\label{eq:lemma:equivalenceQandF}
Q_{\I_t,:}-Q_{\I_{t-1},:}=(F_{\I_t,:}-F_{\I_{t-1},:})(I-\bar{C}SF)^{-1}.
\end{equation}
But $(I-\bar{C}SF)^{-1}$ is invertible so $F_{\I_t,:}-F_{\I_{t-1},:}=0_{n_u\times Tn_y}$ if and only if $Q_{\I_t,:}-Q_{\I_{t-1},:}=0$.

Finally, thanks to (\ref{eq:Q_param:r}), one can write
$$
r_{\I_t}-r_{\I_{t-1}}=f_{\I_t}-f_{\I_{t-1}}+(Q_{\I_t,:}-Q_{\I_{t-1},:})\bar{C}Sf,
$$
where, thanks to the observation following (\ref{eq:lemma:equivalenceQandF}), the last term is cancelled in both the necessary and the sufficient cases, i.e., when $Q_{\I_t,:}-Q_{\I_{t-1},:}=0$ and when $F_{\I_t,:}-F_{\I_{t-1},:}=0$. It leads to $F_{\I_t,:}-F_{\I_{t-1},:}=0$ and $f_{\I_t}-f_{\I_{t-1}}=0$ if and only if $Q_{\I_t,:}-Q_{\I_{t-1},:}=0$ and $r_{\I_t}-r_{\I_{t-1}}=0$.

The case $t=0$ is similar, indeed, consider that $u_{-1}, f_{-1}, F_{(-1,\tau)}$ and all quantities indexed by $\I_{-1}$ are zeros with compatible dimensions.

%% file: main.bbl
\begin{thebibliography}{10}
\providecommand{\url}[1]{#1}
\csname url@samestyle\endcsname
\providecommand{\newblock}{\relax}
\providecommand{\bibinfo}[2]{#2}
\providecommand{\BIBentrySTDinterwordspacing}{\spaceskip=0pt\relax}
\providecommand{\BIBentryALTinterwordstretchfactor}{4}
\providecommand{\BIBentryALTinterwordspacing}{\spaceskip=\fontdimen2\font plus
\BIBentryALTinterwordstretchfactor\fontdimen3\font minus
  \fontdimen4\font\relax}
\providecommand{\BIBforeignlanguage}[2]{{%
\expandafter\ifx\csname l@#1\endcsname\relax
\typeout{** WARNING: IEEEtran.bst: No hyphenation pattern has been}%
\typeout{** loaded for the language `#1'. Using the pattern for}%
\typeout{** the default language instead.}%
\else
\language=\csname l@#1\endcsname
\fi
#2}}
\providecommand{\BIBdecl}{\relax}
\BIBdecl

\bibitem{zhang2020networked}
X.~{Zhang}, Q.~{Han}, X.~{Ge}, D.~{Ding}, L.~{Ding}, D.~{Yue}, and C.~{Peng},
  ``Networked control systems: a survey of trends and techniques,''
  \emph{IEEE/CAA Journal of Automatica Sinica}, vol.~7, no.~1, pp. 1--17, 2020.

\bibitem{schenato2007foundations}
L.~{Schenato}, B.~{Sinopoli}, M.~{Franceschetti}, K.~{Poolla}, and S.~S.
  {Sastry}, ``Foundations of control and estimation over lossy networks,''
  \emph{Proceedings of the IEEE}, vol.~95, no.~1, pp. 163--187, Jan 2007.

\bibitem{heemels2010networked}
W.~M.~H. Heemels, A.~R. Teel, N.~Van~de Wouw, and D.~Ne{\v{s}}i{\'c},
  ``Networked control systems with communication constraints: Tradeoffs between
  transmission intervals, delays and performance,'' \emph{IEEE Transactions on
  Automatic control}, vol.~55, no.~8, pp. 1781--1796, 2010.

\bibitem{branicky2002scheduling}
M.~S. {Branicky}, S.~M. {Phillips}, and {Wei Zhang}, ``Scheduling and feedback
  co-design for networked control systems,'' in \emph{Proceedings of the 41st
  IEEE Conference on Decision and Control, 2002.}, vol.~2, 2002, pp. 1211--1217
  vol.2.

\bibitem{trivellato2010state}
M.~{Trivellato} and N.~{Benvenuto}, ``State control in networked control
  systems under packet drops and limited transmission bandwidth,'' \emph{IEEE
  Transactions on Communications}, vol.~58, no.~2, pp. 611--622, 2010.

\bibitem{heemels2012introduction}
W.~P. M.~H. {Heemels}, K.~H. {Johansson}, and P.~{Tabuada}, ``An introduction
  to event-triggered and self-triggered control,'' in \emph{2012 IEEE 51st IEEE
  Conference on Decision and Control (CDC)}, 2012, pp. 3270--3285.

\bibitem{zhang2017overview}
X.~{Zhang}, Q.~{Han}, and B.~{Zhang}, ``An overview and deep investigation on
  sampled-data-based event-triggered control and filtering for networked
  systems,'' \emph{IEEE Transactions on Industrial Informatics}, vol.~13,
  no.~1, pp. 4--16, 2017.

\bibitem{DEPERSIS2013parsimonious}
\BIBentryALTinterwordspacing
C.~{De Persis}, R.~Sailer, and F.~Wirth, ``Parsimonious event-triggered
  distributed control: A zeno free approach,'' \emph{Automatica}, vol.~49,
  no.~7, pp. 2116--2124, 2013. [Online]. Available:
  \url{https://www.sciencedirect.com/science/article/pii/S0005109813001726}
\BIBentrySTDinterwordspacing

\bibitem{lessard2011quadratic}
L.~Lessard and S.~Lall, ``Quadratic invariance is necessary and sufficient for
  convexity,'' in \emph{Proceedings of the 2011 American Control
  Conference}.\hskip 1em plus 0.5em minus 0.4em\relax IEEE, 2011, pp.
  5360--5362.

\bibitem{rotkowitz2006characterization}
M.~Rotkowitz and S.~Lall, ``A characterization of convex problems in
  decentralized control,'' \emph{IEEE transactions on Automatic Control},
  vol.~51, no.~2, 2006.

\bibitem{gao2006active}
Z.~Gao, ``Active disturbance rejection control: a paradigm shift in feedback
  control system design,'' in \emph{2006 American control conference}.\hskip
  1em plus 0.5em minus 0.4em\relax IEEE, 2006, pp. 7--pp.

\bibitem{wang2019systemlevel}
Y.~{Wang}, N.~{Matni}, and J.~C. {Doyle}, ``A system-level approach to
  controller synthesis,'' \emph{IEEE Transactions on Automatic Control},
  vol.~64, no.~10, pp. 4079--4093, 2019.

\bibitem{BENTAL19991}
\BIBentryALTinterwordspacing
A.~Ben-Tal and A.~Nemirovski, ``Robust solutions of uncertain linear
  programs,'' \emph{Operations Research Letters}, vol.~25, no.~1, pp. 1--13,
  1999. [Online]. Available:
  \url{https://www.sciencedirect.com/science/article/pii/S0167637799000164}
\BIBentrySTDinterwordspacing

\bibitem{sadraddini2020robust}
S.~Sadraddini and R.~Tedrake, ``Robust output feedback control with guaranteed
  constraint satisfaction,'' in \emph{Proceedings of the 23rd International
  Conference on Hybrid Systems: Computation and Control}, 2020, pp. 1--10.

\bibitem{skaf2010design}
J.~Skaf and S.~P. Boyd, ``Design of affine controllers via convex
  optimization,'' \emph{IEEE Transactions on Automatic Control}, vol.~55,
  no.~11, pp. 2476--2487, 2010.

\bibitem{youla1976modern}
D.~{Youla}, H.~{Jabr}, and J.~{Bongiorno}, ``Modern wiener-hopf design of
  optimal controllers--part ii: The multivariable case,'' \emph{IEEE
  Transactions on Automatic Control}, vol.~21, no.~3, pp. 319--338, 1976.

\bibitem{jungers2018observability}
R.~M. {Jungers}, A.~{Kundu}, and W.~P. M.~H. {Heemels}, ``Observability and
  controllability analysis of linear systems subject to data losses,''
  \emph{IEEE Transactions on Automatic Control}, vol.~63, no.~10, pp.
  3361--3376, 2018.

\bibitem{rutledge2020finite}
K.~Rutledge, S.~Z. Yong, and N.~Ozay, ``Finite horizon constrained control and
  bounded-error estimation in the presence of missing data,'' \emph{Nonlinear
  Analysis: Hybrid Systems}, vol.~36, p. 100854, 2020.

\bibitem{hennet1989extension}
J.-C. Hennet, ``Une extension du lemme de farkas et son application au probleme
  de r{\'e}gulation lin{\'e}aire sous contraintes,'' \emph{CR Acad. Sci.
  Paris}, vol. 308, no.~I, pp. 415--419, 1989.

\bibitem{gurobi}
\BIBentryALTinterwordspacing
L.~Gurobi~Optimization, ``Gurobi optimizer reference manual,'' 2021. [Online].
  Available: \url{http://www.gurobi.com}
\BIBentrySTDinterwordspacing

\bibitem{bharadwaj2018synthesis}
S.~Bharadwaj, R.~Dimitrova, and U.~Topcu, ``Synthesis of surveillance
  strategies via belief abstraction,'' in \emph{2018 IEEE Conference on
  Decision and Control (CDC)}.\hskip 1em plus 0.5em minus 0.4em\relax IEEE,
  2018, pp. 4159--4166.

\bibitem{posa2017balancing}
M.~A. Posa, T.~Koolen, and R.~L. Tedrake, ``Balancing and step recovery
  capturability via sums-of-squares optimization,'' 2017.

\bibitem{yang2018comparison}
X.~Yang and J.~K. Scott, ``A comparison of zonotope order reduction
  techniques,'' \emph{Automatica}, vol.~95, pp. 378--384, 2018.

\end{thebibliography}
